%% file: main.tex
\documentclass[12pt]{article} 

\usepackage{amsmath,amssymb,amsthm}
\usepackage[utf8]{inputenc}
\usepackage{tcolorbox}
\usepackage{mathtools}
\usepackage{unicode}
\usepackage{graphicx}
\usepackage{hyperref}
\usepackage{comment}
\usepackage{braket}
\usepackage{dsfont}
\usepackage{cleveref}
\usepackage{thmtools}
\usepackage{nameref}
\usepackage{enumitem}
\usepackage{cancel}
\usepackage{chngcntr} 
\usepackage{xcolor}
\usepackage{todonotes}

\newtheorem{thm}{Theorem}[section]
\newtheorem*{claim*}{Claim}
\newtheorem{cor}[thm]{Corollary}
\newtheorem*{cor*}{Corollary}
\newtheorem{lemma}[thm]{Lemma}
\newtheorem*{lemma*}{Lemma}
\newtheorem{dfn}[thm]{Definition}
\newtheorem*{dfn*}{Definition}
\newtheorem{prop}[thm]{Proposition}
\newtheorem*{prop*}{Proposition}
\theoremstyle{definition}
\newtheorem{rmk}[thm]{Remark}
\AtEndEnvironment{rmk}{\qed}%
\newtheorem*{rmk*}{Remark}
\AtEndEnvironment{rmk*}{\qed}%
\newtheorem{example}[thm]{Example}
\AtEndEnvironment{example}{\qed}%
\newtheorem*{example*}{Example}
\AtEndEnvironment{example*}{\qed}%
\newtheorem{notation}[thm]{Notation}
\AtEndEnvironment{notation}{\qed}%
\newtheorem*{notation*}{Notation}
\AtEndEnvironment{notation*}{\qed}%

\usepackage{thmtools}

\usepackage{chngcntr}
\counterwithout{equation}{section}

\newcommand{\ie}{i.e.}

\newcommand{\rngmlt}{\ast} 
\newcommand{\strmlt}{\star} 
\newcommand{\brmlt}{\circ}   
\newcommand{\fctrmlt}{\odot} 

\newcommand{\sgm}[2]{\sigma_{#1}(#2)}
\newcommand{\tauu}[2]{\tau_{#1}(#2)}

\newcommand{\prm}{^{\prime}}
\newcommand{\dprm}{^{\prime \prime}}

\newcommand{\ints}{\mathds{Z}}
\newcommand{\nats}{\mathds{N}}
\newcommand{\complex}{\mathds{C}}
\newcommand{\gequiv}{$\mathcal{G} (X,r)-$equivariant}
\newcommand{\gequivB}{$\mathcal{G} (B,r)-$equivariant}
\newcommand{\setGequiv}{\widetilde{\mathcal G } (X,r)}
\newcommand{\setGequivB}{\widetilde{\mathcal G } (B,r)}
\newcommand{\id}{\text{id}}
\newcommand{\cnst}{\text{cnst}}

\newcommand{\aring}{(R,+,\rngmlt)}
\newcommand{\nring}{(N,+,\rngmlt)}
\newcommand{\abrace}{(B,+,\brmlt)}

\newcommand{\nbrace}{(N,+,\brmlt)}
\newcommand{\nfctr}{(N,+,\fctrmlt)}
\newcommand{\soc}[1]{\text{Soc}( #1 ) }
\newcommand{\tildek}{\widetilde{k}}
\newcommand{\hatk}{\widehat{k}}
\newcommand{\op}{\text{op}}

\makeatletter
\def\mathcolor#1#{\@mathcolor{#1}}
\def\@mathcolor#1#2#3{%
  \protect\leavevmode
  \begingroup
    \color#1{#2}#3%
  \endgroup
}
\makeatother

\makeatletter
\newcommand*{\inlineequation}[2][]{%
  \begingroup
    \refstepcounter{equation}%
    \ifx\\#1\\%
    \else
      \label{#1}%
    \fi
    \relpenalty=10000 %
    \binoppenalty=10000 %
    \ensuremath{
      #2%
    }
    ~\@eqnnum
  \endgroup
}
\makeatother

\makeatletter
\DeclareRobustCommand{\em}{%
  \@nomath\em \if b\expandafter\@car\f@series\@nil
  \normalfont \else \bfseries \fi}
\makeatother

\newcommand\restr[2]{{
  \left.\kern-\nulldelimiterspace 
  #1 
  \vphantom{\big|} 
  \right|_{#2} 
  }}

\title{New Solutions to the Reflection Equation with Braces}
\author{Kyriakos Katsamaktsis\\ University of Oxford\footnote{The work was carried out while the author was affiliated with the University of Edinburgh.}}
\date{\today}

\begin{document}

\maketitle




\begin{abstract}
We find several new solutions to the set-theoretic analogue of the celebrated reflection equation of Cherednik. We use braces, an algebraic structure associated to the related Yang-Baxter equation, by building on the recent work of Smoktunowicz, Vendramin and Weston. We show that every brace with a nontrivial centre yields several simple solutions and that a natural class of solutions is a near-ring. We find more solutions for factorizable rings. Some of our solutions apply to the original parameter-dependent equation with rational parameter dependence.
\end{abstract}




\input{introduction}

\input{preliminaries}

\input{reflections_for_braces}

\input{reflections_for_factorizable_rings}

\input{parameter_dependent}



\bibliographystyle{plain} 

\bibliography{bibliography} 


\end{document}

%% file: introduction.tex
\section{Introduction}
%
%
The well-known Yang-Baxter equation~\cite{ybe-fst1,ybe-fst2} has intrigued mathematicians and physicists for over five decades. Since it first appeared in problems in statistical mechanics~\cite{ybe-fst2} and quantum scattering~\cite{ybe-fst1},  connections to many fields of pure and applied mathematics have been drawn.
In mathematical physics it occupies a central role in the theory of integrable systems~\cite{jimbo}, while
on the mathematical side it played an important part in the development of Hopf algebras~\cite{hopf, hopf2} and quantum groups~\cite{icm, kassel}.
%
%
%
%
%
%
%
%
Since Drinfeld's suggestion in the 1986 ICM~\cite{icm} much research has focused on the set-theoretic variant of the equation, the study of which benefited from the seminal works of Etingof, Schedler and Soloviev~\cite{etingof} and Gateva-Ivanova and Van den Bergh~\cite{ivanova}. 


Closely related to the Yang-Baxter equation is the reflection equation which was introduced by Cherednik~\cite{fst-refl} in the study of quantum scattering on the half line.
The reflection equation and associated algebraic structures are relevant to many fields of mathematics and physics.
Similarly to the Yang-Baxter equation, the reflection equation is important in the theory of quantum groups~\cite{kulish1992algebraic, molev1996yangians}
and integrable systems~\cite{sklyanin}.
Kulish~\cite{kulish1993covariance} mentions connections to 
non-commutative differential geometry~\cite{Alekseev1991,schirrmacher1991two,kobayashi1992differential}
and lattice Kac-Moody algebras~\cite{alekseev1992hidden, babelon1991universal,freidel1991quadratic, freidel1991classical}.
Researchers have previously obtained solutions to the reflection equation using Hecke algebras \cite{doikou-martin02,doikou2005affine} and the Temperley–Lieb algebra \cite{avan2011reflection}.

Caudrelier and Zhang~\cite{fst-comb-refl} recently formulated the set-theoretic reflection equation and provided the first examples of solutions, called reflections. 
In~\cite{syst-comb-refl} the set-theoretic reflection equation was first studied systematically and the reflections for a specific class of solutions of the Yang-Baxter equation were classified.

In this paper we describe new solutions to the set-theoretic reflection equation using the theory of braces.
Braces, a generalisation of Jacobson radical rings, were introduced by Rump~\cite{rump} in a seminal paper to study the set-theoretic Yang-Baxter equation, where he showed that every non-degenerate involutive solution to the set-theoretic Yang-Baxter equation can be embedded into a brace. 
Given the relationship between the reflection and the Yang-Baxter equations, searching for reflections coming from braces is a natural approach. 
The definition of a brace was further crystallised by Ced{\'o}, Jespers and Okniński~\cite{cedo} and noninvolutive solutions are described by skew braces~\cite{skew}. 
The last few years have seen an explosive growth in the number of publications on the Yang-Baxter equation and braces~\cite{cedo, cedo2, cedo-surv1, cedo-surv2, cedo-surv3, cedo-surv4, cedo-surv5, cedo-surv6, cedo-surv7}.
Our work builds on the recent results of Smoktunowicz, Vendramin and Weston~\cite{refl} who first used a brace-theoretic approach to the reflection equation.

The paper is organised as follows. 
In section \ref{section-preliminaries} we give the necessary background to braces and the reflection equation and describe some simple examples of reflections. 
In section \ref{section-reflection-for-braces} we find several simple reflections for an arbitrary brace and show that the class of \gequiv\ maps, a subset of the reflections first introduced in~\cite{refl}, is a near-ring. 
Next in section \ref{section-reflections-for-factorizable-rings} we find reflections for factorizable rings, a specific class of braces coming from nilpotent rings. 
Lastly in section \ref{section-parameter-dependent} we show that some of the reflections from the previous sections yield solutions to the parameter-dependent reflection equation that is used in physics.


%% file: preliminaries.tex
\section{Preliminaries} \label{section-preliminaries}




\subsection{Introducing the Reflection Equation}

Recall the (linear) Yang-Baxter equation (YBE):
\begin{equation} \label{linear-ybe}
    (R \otimes I) (I \otimes R) (R \otimes I) = (I \otimes R) (R \otimes I) (I \otimes R)
\end{equation}
where $R\colon V\otimes V \to V\otimes V$ for $V$ a vector space and $I$ is the identity operator on $V$. The set-theoretic YBE is
\[
(\id \times r) (r \times \id) (\id \times r)
= (r \times \id) (\id \times r) (r \times \id) 
\]
where $r:X\times X \to X\times X$ for a set $X$, $\id$ is the identity map on $X$, and the operation of the maps on either side is function composition. We write
$r(x,y) = (\sigma_x(y), \tau_y(x))$.
A solution to the set-theoretic YBE corresponds to a solution to the linear YBE that only permutes the basis vectors. 
We call $(X,r)$ \emph{involutive} if 
$r^2= \id$ 
and \emph{nondegenerate} if 
$\sigma_x,\tau_y$ 
are bijective $X\to X$ maps for every $x,y \in X$.

Related to the Yang-Baxter equation is the reflection equation:
\begin{dfn}[Equation (10), \cite{fst-refl}] \label{refl-linear-def}
Let $V$ be a vector space and
\(
R: V\otimes V  \to V\otimes V 
\) 
a solution to the Yang-Baxter equation for $V$. Then a map
\(
K: V \to V
\) 
is a solution to the \emph{reflection equation} if it satisfies
\begin{equation} \label{refl-linear-eq}
    R (I \otimes K) R (I \otimes K) = (I \otimes K) R (I \otimes K) R.
\end{equation}
\end{dfn}

As with the Yang-Baxter equation, we can define a set-theoretic analogue of the reflection equation:
\begin{dfn}[Definition 1.1, \cite{fst-comb-refl}]
Let $(X,r)$ be a non-degenerate set theoretic solution to the YBE. A map $k : X \to X$ is a set-theoretic solution to the reflection equation for $(X,r)$ if it satisfies
\begin{equation} \label{refl-eq}
    r (\id \times k) r (\id \times k) = (\id \times k) r (\id \times k) r.
\end{equation}
We also say that $k$ is a \emph{reflection} of $(X,r)$. 
\end{dfn}

\begin{example} \label{refl-eq-ex-mod-n}
Consider $(X,r)$ with $X=\ints_n$ and $r(x,y) = (y+1,x)$. It is easy to see that $r$ is a solution of the set-theoretic YBE. Let $(a,b) \in X \times X$ and $k: X \to X$ be arbitrary. Then the RHS of \eqref{refl-eq} evaluates to
\[
(k(a)+1, k(b+1))
\]
and the LHS evaluates to
\[
    (k(a)+1, k(b) +1).
\]
Hence $k$ is a reflection if and only if
\[
    k(b+1) = k(b) + 1, b\in X
\]
equivalently
\[
    k(1) = k(0) + 1,\ k(2) = k(0) + 2,...,k(n-1) = k(0) + n-1
\]
hence there are precisely $n$ reflections, each one fixed by the value of $k(0)$.
\end{example}

\begin{rmk}
In theorem 1.8 of~\cite{refl} it is shown that if $(X,r)$ is a nondegenerate involutive solution to the set-theoretic YBE then to check whether a map $k$ is a reflection of $(X,r)$ it suffices to check that only the first coordinates of the reflection equation match. Example \ref{refl-eq-ex-mod-n} shows that the assumption that $r$ is involutive is necessary.
\end{rmk}

\begin{example} \label{refl-eq-ex-mod-2}
If we take $n=2$ in example \ref{refl-eq-ex-mod-n} the only reflections are the identity and the transposition $(12)$.
\end{example}
\noindent The following example is well-known in the physics community:
\begin{example} \label{refl-eq-ex-invertible} 
Consider the non-degenerate set theoretic solution $(X,r)$ with $r(x,y) = (\phi (y) , \phi ^{-1} (x))$ where $\phi : X \to X$ is a bijection. Then for arbitrary $(a,b) \in X \times X$ and $k:X \to X$ both sides of equation \eqref{refl-eq} evaluate to $(\phi k \phi ^{-1} (a),k(b))$ as the following calculation shows:
\begin{align*}
    (a,b) &{\overset{k_2}{\longmapsto}} (a,k(b)) \\
          &{\overset{r}{\longmapsto}} (\phi k(b), \phi^{-1}(a))\\
          &{\overset{k_2}{\longmapsto}} (\phi k(b), k \phi^{-1}(a) )\\
          &{\overset{r}{\longmapsto}} (\phi k \phi ^{-1} (a),k(b)).
\end{align*}
And for the other side of \eqref{refl-eq} we have
\begin{align*}
    (a,b)  &{\overset{r}{\longmapsto}} (\phi(b),\phi^{-1}(a)) \\
          &{\overset{k_2}{\longmapsto}} (\phi (b), k \phi^{-1}(a))\\
          &{\overset{r}{\longmapsto}} (\phi k \phi^{-1}(a), b )\\
          &{\overset{k_2}{\longmapsto}} (\phi k \phi ^{-1} (a),k(b)).
\end{align*}
Hence any map $k: X \to X$ is a reflection.
\end{example}




\subsection{Braces} \label{subsection-braces}
Braces were introduced by Rump \cite{rump} to study the YBE. The definition we present here was distilled in \cite{cedo}.

\begin{dfn}[Definition 2.2 \cite{cedo}] \label{dfn-brc}
Let $B$ be a nonempty set with two binary operations, $+$, called addition, and $\brmlt$, called multiplication. We call 
$(B, +, \brmlt)$ a \emph{left brace} or simply \emph{brace} if
\begin{enumerate}
    \item $(B,+)$ is an abelian group, called the additive group of the brace, with identity $0\in B$
    \item $(B,\brmlt)$ is a group, called the multiplicative group
    \item for $x,y,z \in B,$
    \begin{equation} \label{dstr}
        x \brmlt (y + z) = x \brmlt y + x\brmlt z -x.
    \end{equation}
\end{enumerate}
For $x\in B$ we denote the inverse of $x$ in $(B,+)$ by $-x$ and the inverse in $(B,\brmlt)$ by $x^{-1}$.
\end{dfn}
\noindent A \emph{subbrace} is a subset of $\abrace$ that satisfies the axioms of a brace with the same operations.

Rump showed in~\cite{rump} that every nondegenerate, involutive set-theoretic solution to the YBE can be embedded in a brace. The next theorem is implicit in \cite{rump} and explicit in \cite{cedo}.
\begin{thm}\label{yb-brc-map}
Let $(B,+,\brmlt)$ be a brace and for $x,y \in B$ define $r(x,y) = (\sigma_x(y), \tau_y(x))$, the \emph{Yang-Baxter map} of $\abrace$, where
\begin{equation}
    \sigma_x(y) = x\brmlt y - x,\ \tau_y(x) = (\sigma_x(y))^{-1} \brmlt x - (\sigma_x(y))^{-1}.
\end{equation}
Then $(B,r)$ is a nondegenerate involutive solution of the set-theoretic YBE. 

Conversely, for every nondegenerate involutive solution of the set-theoretic YBE $(X,r)$, there is a brace $\abrace$ such that $X\subseteq B$ and $r$ is the restriction of the Yang-Baxter map of $\abrace$ to $X\times X$.
\end{thm}

\begin{notation}
In several parts of the paper we will encounter results that start as ``Let $\abrace$ be a brace with Yang-Baxter map $r$ and a subset $X\subseteq B$ such that $(X,r)$ is a nondegenerate set-theoretic solution to the YBE''. By $(X,r)$ strictly speaking what is meant is $(X,\restr{r}{X\times X})$. For the sake of brevity and to keep notation simple we write $r$ instead of $\restr{r}{X\times X}$.
\end{notation}

At several points we will use the following basic identity:
\begin{lemma}[pp. 4, \cite{cedo}] \label{cedo-lemma}
Let $(B,+,\brmlt)$ be a brace and $x,y,z \in B.$ Then
\[
z \brmlt (x-y) - z= z\brmlt x - z\brmlt y
\]
\end{lemma}

\noindent We will also make use of the following simple identities:
\begin{lemma} \label{brcmlt-ids}
Let $\abrace$ be a brace, $x,y,y_1,...,y_n\in B$ and $n\in \nats$ \footnote{by $\nats$ we mean the integers $\geq 1$.}. 
Then the following are true:
\begin{enumerate}
    \item 
    $x\brmlt y + x\brmlt (-y) = 2x$.
    \item 
    $x\brmlt(y_1 + \cdot \cdot \cdot + y_n) = x\brmlt y_1 + x\brmlt y_2 + \cdot \cdot \cdot + x\brmlt y_n - (n-1) x$.
    \item 
    $x\brmlt(ny) = n\ x\brmlt y - (n-1) x$.
    \item
    $x\brmlt (-ny) = (n+1) x -n\ x\brmlt y$
\end{enumerate}
\end{lemma}
\begin{proof}
\hfill
\begin{enumerate}
    \item 
     Using distributivity~\eqref{dstr} we have:
    \[
         x = x\brmlt 0 
        = x\brmlt (y + (-y))
        = x\brmlt y + x\brmlt (-y) - x
    \]
         and rearranging gives the desired result.
    \item 
         This follows from distributivity~\eqref{dstr} by induction.
    \item This follows from identity 2 on the list by setting $y_1 = y_2 = ... = y_n = y$.
    
    \item This follows from identities 1 and 3 on the list.
\end{enumerate}
\end{proof}

Straightforward calculations prove the following well-known identities. The proof of the first two can be found in lemma 2.6 of \cite{cedo2}.
\begin{prop}\label{sigma-homo}
Let $(B,+,\brmlt)$ be a brace with Yang-Baxter map $r(x,y) = (\sigma_x(y), \tau_y(x))$ and $X\subseteq B$ so that $(X,r)$ is a nondegenerate involutive solution of the set-theoretic YBE. Then:
\begin{enumerate}
    \item $\sigma_x$ is a $(B,+)$ group homomorphism for any $x\in X$.
    \item $\inlineequation[sigma-compose]{
     \sigma_x  \sigma_y = \sigma_{x\brmlt y},\ x,y \in X 
    }$
    \item  $\inlineequation[brc-sigma-tau]{
    \sigma_x(y) \brmlt \tau_y (x) = x \brmlt y,\ x,y \in X}$
\end{enumerate}
\end{prop}

\noindent Similarly, we can define \emph{right braces} by replacing condition \eqref{dstr} by 
\begin{equation} \label{right-brc-dstr}
    (x+y) \brmlt z = x\brmlt z + y\brmlt z -z.
\end{equation}

\begin{dfn}
Let $\abrace$ be a left brace.
The \emph{opposite brace} of $\abrace$, denoted by
$(B^{\text{op}}, +, \brmlt)$,
is the right brace with underlying set and additive group the same as $\abrace$ and with multiplicative group the opposite of the multiplicative group of $\abrace$.
\end{dfn}

\noindent If for a brace both \eqref{dstr} and \eqref{right-brc-dstr} hold it is called a \emph{two-sided} brace.

Two important examples of braces are nilpotent rings and the wider class of Jacobson radical rings. 
Recall that a ring $\nring$ is nilpotent if and only if there is a $n\in \nats$  such that $a_1 \rngmlt a_2 \rngmlt ... \rngmlt a_n = 0$ for any $a_1,a_2,...,a_n \in N$. 
A ring $\aring$ is Jacobson radical if and only if for any $x\in R$ there is a unique $y\in R$ such that $x+y +x\rngmlt y = 0$. 
Another characterisation of Jacobson radical rings is that, when $R$ is embedded into a ring with multiplicative identity $1_{\rngmlt}$, the elements of the form $r+1_{\rngmlt},\ r\in R,$ have a $\rngmlt$-inverse of the same form. 
Jacobson radical rings are braces with multiplication the adjoint multiplication of the ring $x\brmlt y := x+y+x\rngmlt y$ and addition the ring addition. 
It is not hard to see that all two-sided braces are actually Jacobson radical rings.

Motivated by ring multiplication, we consider the following operation on braces:

\begin{notation}
Let $\abrace$. For $x,y \in B$ define
\[
x\strmlt y := x\brmlt y - x -y.
\]
\end{notation}

Obviously $\strmlt = \rngmlt$ for a Jacobson radical ring $\aring$. 

Since the appearance of braces, it was not known whether $\strmlt$ is associative if the brace is one-sided. Lau recently showed that this is not the case \cite{ivan}; $\strmlt$ is associative only if the brace is a Jacobson radical ring.

A brace is called \emph{trivial} if multiplication coincides with addition.
The \emph{socle} of a brace is the largest subset for which multiplication from the right by any element of the brace is trivial.
\begin{dfn} \label{dfn-soc}
Let $\abrace$ be a brace. The \emph{socle} of $B$ is
\[
\soc{B} := \{ a\in B : a\brmlt b = a+b,
              \text{ for all } b\in B       \}.
\]
\end{dfn}


\begin{prop}\label{soccle-fixpoint}
Let $\abrace$ be a left brace, $r$ be the Yang-Baxter map of $\abrace$ and $X\subseteq B$ such that $(X,r)$ is a solution to the YBE. Then $z\in X$ is a fixpoint of $\sigma_x$ for every $x\in X$ if and only if 
$z\in X \cap \soc{B^{\text{op}}}$.
\end{prop}
\begin{proof}
Since $\sigma_x (z) = x\brmlt z - x$, we see that $z$ is a fixpoint for every $x\in X$ if and only if $x\brmlt z = x+z$ for each $x\in X$, \ie\ 
$z\in X \cap \soc{B^{\text{op}}}.$
\end{proof}

For more on braces and their relationship to the Yang-Baxter equation the reader is pointed to~\cite{cedo2}.

%% file: reflections_for_braces.tex
\section{Reflections for Braces}\label{section-reflection-for-braces}

A particular class of maps, first identified in \cite{refl}, gives a very simple criterion for being a reflection. 
\begin{dfn}
Let $(X,r),\ r(x,y) = (\sigma_x(y), \tau_y(x)),$ be a set-theoretic solution of the YBE. 
A map $k : X \to X$ is called \emph{$\mathcal{G} (X,r)-$equivariant} if $k\sigma_x = \sigma_x k$ for all $x\in X$. 
We denote the set of \gequiv\ maps of $(X,r)$ by $\setGequiv$.
\end{dfn}

\begin{thm}[Theorem 1.9, \cite{refl}] \label{equiv-refl}
Let $(X,r)$ be an involutive non-degenerate set-theoretic solution of the YBE. Every $\mathcal{G} (X,r)-$equivariant map $k : X\to X$ is a solution to the reflection equation.
\end{thm}

Moreover, \gequiv\ maps are in a sense ``building blocks'' for more complicated reflections (cf. theorem~\ref{refl-eq-thm-sols-equiv}).

In this section we obtain \gequiv\ maps for an arbitrary brace, hence obtaining reflections of the brace by \ref{equiv-refl}. 

First we observe that the \gequivB\ maps of a brace $\abrace$ form a right near-ring.

\begin{dfn} \label{dfn-near-ring}
Let $S$ be a set with two operations $+$ and $\cdot$. Then we call $(S,+,\cdot)$ a \emph{right near-ring} if
\begin{enumerate}
    \item $(S,+)$ is a group;
    \item $(S,\cdot)$ is a semi-group;
    \item $(x+y)\cdot z = x\cdot z + y\cdot z$ for all $x,y,z \in S$.
\end{enumerate}
\end{dfn}
We remind the reader that $(S,\cdot)$ is called a semi-group when the operation $\cdot$ is associative. Thus a right-near ring is the same as a ring, with the left-distributivity assumption dropped.

\begin{thm} \label{gequiv-near-ring}
Let $\abrace$ be a brace with Yang-Baxter map $r$. Then $(\setGequivB, +, .)$, where $.$ is function composition and $+$ is pointwise addition of maps, is a right near-ring. Moreover, the unit of $+$ is the zero map and $.$ has as unit the identity map.
\end{thm}
\begin{proof}
It is easy to see that the zero map and the identity map are indeed \gequivB\ and they are the units for $+$ and $.$ respectively.

The fact that $.$ right-distributes over $+$ follows directly from the definitions of composition and addition of maps.

To show $\setGequivB$ is closed under addition, let $f,g \in \setGequivB $ and $x,y \in B$. 

We have:
\begin{align*}
      \sigma_x (f+g)(y) &= \sigma_x (f(y) + g(y))\\
                        &= \sigma_x f(y) + \sigma_x(g(y)) 
                        &&\text{by \ref{sigma-homo}}
                        \\
                        &= f \sigma_x(y) + g \sigma_x(y)
                        &&\text{$f,g$ are \gequivB.}
                        \\
                        &= (f+g) \sigma_x (y)
    \end{align*}
which shows that $\setGequivB$ is closed under $+$

To show $\setGequivB$ is closed under $.$ we calculate:
\begin{align*}
        \sigma_x (fg)(y) &= \sigma_x f g(y)
                         \\
                         &= f \sigma_x g (y)
                         &&\text{$f$ is \gequiv}
                         \\
                         &= f g \sigma_x (y)
                         &&\text{$g$ is \gequiv}
                         \\
                         &= (fg) \sigma_x (y)
    \end{align*}
thus $\setGequivB$ is closed under function composition.

Finally to show $-f \in \setGequivB$ observe that the map $x \mapsto -x$ is $\mathcal{G} (B,r)$--equivariant:
\begin{align*}
    \sigma_x (-y) &= x\brmlt (-y) -x \\
                  &= 2x - x\brmlt y -x
                  &&\text{by \ref{brcmlt-ids}}
                  \\
                  &= -(x\brmlt y -x)
                  \\
                  &=- \sigma_x(y);
\end{align*}
thus its composition with $f$, $-f$, is \gequivB.
\end{proof}

\begin{rmk}
In the above theorem instead of $B$ itself we could have considered $X\subseteq B$. The minimal requirements on $X$ are that $X$ is an additive subgroup of $B$ (which is equivalent to $f-g$ being an $X \to X$ map, for $f,g: X\to X$) and that $(X,r)$ is a set-theoretic solution to the YBE (so that $\setGequiv$ makes sense). The latter amounts to $r(X,X) \subseteq X$. But these together imply that $X$ is a subbrace, as the next proposition shows.
\end{rmk}

\begin{prop}
Let $\abrace$ be a brace with Yang-Baxter map $r$ and $X\subseteq B$ an additive subgroup of $B$ such that $r(X,X) \subseteq X$. Then $X$ is a subbrace.
\end{prop}
\begin{proof}
Since $(X,+) \leq (B,+)$ it suffices to show $(X,\brmlt) \leq (B,\brmlt)$. Since the multiplicative and the additive unit of a brace coincide, we just need to show $X$ is closed under $\brmlt$ and multiplicative inverses.

Since $r(X,X) \subseteq X$, for $x,y \in X$ we have $\sigma_x(y) \in X$, whence
\[
x\brmlt y = \sigma_x(y) + x \in X,
\]
which proves the former.

To show closure under inverses, let $a\in X$. For a fixed $x_1\in X$ to be determined later we can find a single $y_1\in X$ such that $\sigma_{x_1}(y_1) = a,$ since $(X,r)$ is a nondegenerate solution.

Then
\begin{align*}
    \tau_{y_1}(x_1) &= (\sigma_{x_1} (y_1) )^{-1} \brmlt x_1 - (\sigma_{x_1} (y_1) )^{-1}
    \\
    &= a^{-1} \brmlt x_1 - a^{-1}
\end{align*}

We know $\tau_{y_1}(x_1) \in X$, hence taking $x_1 = a$ gives $-a^{-1} \in X$ and since $X$ is an additive subgroup we have $a^{-1} \in X$.
\end{proof}

The next theorem gives two particular \gequiv\ maps for an arbitrary $X\subseteq B$.

\begin{thm} \label{k1-k2-equiv}
Let $\abrace$ be a brace, $r$ the Yang-Baxter map associated to the brace and $X\subseteq B$ such that $(X,r)$ is a nondegenerate involutive solution to the YBE. Let $c\in B$ be central in $(B,\brmlt)$.

Then the map
\begin{equation} \label{k1}
    k_1(x) = c\brmlt x -c,\ x\in X
\end{equation}
is \gequiv, assumming $k_1(X) \subseteq X$.

Moreover the map
\begin{equation} \label{k2}
    k_2(x) = c\brmlt x + c, \ x\in X
\end{equation}
is \gequiv, assumming $k_2(X) \subseteq X$, if and only if $2\ c\brmlt x = 2x + 2c,$ for all $x\in X$.
\end{thm}

\begin{proof}
Fix $x\in X$ and let $y\in X$ be arbitrary.

For $k_1$ we have
\begin{align*}
    \sigma_x k_1 (y)
    &= \sigma_x ( c\brmlt y -c )
    \\
    &= x\brmlt ( c\brmlt y -c) - x
    \\
    &= x\brmlt c \brmlt y -x\brmlt c
    &&{\text{  by \ref{cedo-lemma}}}&&
\end{align*}

For $k_1 \sigma_x$ we have
\begin{align*}
    k_1 \sigma_x (y)
    &= k_1 (x\brmlt y - x)
    \\
    &= c\brmlt(x\brmlt y-x) -c 
    \\
    &= c\brmlt x\brmlt y - c\brmlt x
    &&{\text{  by \ref{cedo-lemma}}}&&
\end{align*}
We see that for $c$ central in $(B,\brmlt)$ both sides are equal.

For $k_2$ we proceed similarly:
\begin{align*}
    \sigma_x k_2 (y)
    &= \sigma_x(c\brmlt y +c)
    \\
    &= x\brmlt (c\brmlt y + c) -x
    \\
    &=  x\brmlt c \brmlt y + x\brmlt c - 2x
    &&{\text{  by \eqref{dstr}}}&&
\end{align*}
And for $k_2 \sigma_x (y)$ we have:
\begin{align*}
 k_2 \sigma_x (y) 
 &=  k_2 (x\brmlt y -x) 
 \\
 &= c\brmlt (x\brmlt y -x) + c
 \\
 &= c\brmlt x\brmlt y - c\brmlt x + 2c
 &&{\text{  by \ref{cedo-lemma}}}&&
\end{align*}

Hence $k_2$ is \gequiv\ if and only if
\begin{align*}
   {}&& \sigma_x k_2 (y) &= k_2 \sigma_x (y)
   \\
   {\Leftrightarrow}&&
      x\brmlt c \brmlt y + x\brmlt c -2x
   &= 
      c\brmlt x\brmlt y - c\brmlt x + 2c
   \\
   {\Leftrightarrow}&&
  2\ c\brmlt x &= 2x + 2c,
\end{align*}
for $x\in X$ arbitrary.
\end{proof}

\begin{rmk} \label{refl-central-other-ass-rmk}
The condition  $2\ c\brmlt x = 2x + 2c $ is satisfied if $c \in \soc{B}$. The latter is equivalent to 
\begin{equation} \label{refl-central-other-ass-eq-strmlt}
   c \strmlt  b= 0,\ b\in B.
\end{equation}
\end{rmk}

\begin{example}
Let $(N,+,\rngmlt)$ be a nilpotent ring. Then $(N,+,\brmlt)$, where $\brmlt$ is the adjoint multiplication, is a brace and $\strmlt$ is the ring multiplication.
Since $N$ is nilpotent, we can find $n\in\nats$ such that $a_1 \rngmlt ... \rngmlt a_n = 0$ for all $a_1,...,a_n \in N$. 
Since $\strmlt$ is the ring multiplication, the condition \eqref{refl-central-other-ass-eq-strmlt} is satisfied for $c=a_1 \rngmlt ... \rngmlt a_{n-1}$ for any $a_1,...,a_{n-1} \in B$. 
Let $X\subseteq N$ such that  the Yang-Baxter map $r$ associated to the brace $(N,+,\brmlt)$ when restricted to $X\times X$ makes $(X,r)$ a solution of the set-theoretic YBE. 
Then the maps $k_1(x) = c\brmlt x + c$ and $k_2(x) = c\brmlt x - c$ are reflections of $(X,r)$.
\end{example}

The socle of the brace gives rise to simple reflections:
\begin{prop} \label{socle-equiv}
Let $\abrace$ be a brace with Yang-Baxter map $r$ and $a\in \soc{B^{\op}}$. Then the constant map
\begin{equation}
    \cnst_a (x) = a
\end{equation}
is \gequivB.

Moreover, for $c\in B$ central in $(B,\brmlt)$ the map
\[
m_c (x) = c\brmlt x,\ x\in B
\]
is \gequivB\ if and only if $c\in \soc{B}$.
\end{prop}
\begin{proof}
Since $a\in \soc{B^\op}$, we have $b\brmlt a = a+b$ for any $b\in B$. Thus for $x,y \in B$ arbitrary we have:
\begin{equation*}
    \sigma_x \cnst_a (y) = \sigma_x (a) = x\brmlt a - x = x+a - x = a = \cnst_a \sigma_x (y)
\end{equation*}

For the latter claim for $x,y \in B$ arbitrary we have:
\[
\sigma_x m_c(y) = \sigma_x(c\brmlt y) = x\brmlt c \brmlt y - x
\]
and 
\[
m_c \sigma_x (y) 
= m_c (x\brmlt y - x) 
= c \brmlt (x\brmlt y -x)
= c\brmlt x \brmlt y - c\brmlt x + c,
\]
where for the last equality we used \ref{cedo-lemma}.

Using that $c$ is central, we see the two expressions are equal if and only if 
$-x = - c\brmlt x + c$ 
\ie\ 
$c\in \soc{B}$.
\end{proof}

In view of the closure under pointwise addition and function composition of $\setGequivB$, we can obtain several more solutions to the reflection equation.


\begin{cor}
Let $\abrace$ be a brace with Yang-Baxter map $r$ and let $c\in B$ be central in $(B,\brmlt)$. Then the following maps are \gequivB:

\begin{enumerate}
    \item 
    \begin{equation} \label{k1n}
        k_{1,n} (x) = c\brmlt x -c + n\ x,\ n\in \ints
    \end{equation}
    
    If we moreover assume that
    \[
    2\ c\brmlt b = 2b + 2c, \text{ for all } b\in B,
    \]
    then the following maps are also \gequivB:
    \item
    \begin{equation} \label{k2n}
        k_{2,n} (x) = c\brmlt x + c + n\ x,\ n\in \ints
    \end{equation}
    
    \item 
    \begin{equation}
        \widetilde{k}_m (x)= 
            c^{2m} \brmlt x + c^{2m}
            + 2\  \sum_{j=1}^{m-1} c^{2j}
            - 2\ \sum_{j=1}^{m} c^{2j-1}
            ,\ m\in \nats
    \end{equation}
    
    \item
    \begin{equation}
        \hatk_m (x)= 
            c^{2m} \brmlt x - c^{2m}
            - 2\  \sum_{j=1}^{m-1} c^{2j}
            + 2\ \sum_{j=1}^{m} c^{2j-1}
            ,\ m\in \nats  
    \end{equation}

    \item
    \begin{equation} \label{lm}
        l_{m,n} (x) = 2m c + (2m+n) x,\ m \in \nats, n\in \ints 
    \end{equation}
    
\end{enumerate}
\end{cor}
\begin{proof}
\hfill
\begin{enumerate}

\item 
This is just the sum of $k_1$ with $n$ copies of $\id$ for $n\geq 0$ ($k_{1,0} = k_1$) and the sum of $k_1$ with $-n$ copies of the function $(x\mapsto -x)$ for $n<0$. Since all summands are in $\setGequivB$, their sum is also in $\setGequivB$, by virtue of theorem \ref{gequiv-near-ring}.

\item 
This is just the sum of $k_2$ with $n$ copies of $\id$ for $n\geq 0$ ($k_{2,0} = k_2$) and the sum of $k_2$ with $-n$ copies of $(x\mapsto -x)$ for $n<0$.

\item 
Let $\tildek = k_1 k_2$. We will show that $\widetilde{k}_m (x)= (\tildek)^m (x)$.

For the base case we have:
\begin{align*}
    \tildek (x) = k_1 k_2 (x) &= c \brmlt (c\brmlt x + c) - c 
                \\
                &= c^2 \brmlt x + c^2 - 2c && \text{ by \eqref{dstr}}
                \\
                &= \widetilde{k}_1 (x).
\end{align*}
And for the inductive step:
\begin{align*}
   \tildek^{m+1} (x) &= 
                     \tildek \tildek ^{m} (x)
                     \\
                     &= 
                     \tildek \big( c^{2m} \brmlt x + c^{2m}
                     + 2\  \sum_{j=1}^{m-1} c^{2j}
                      - 2\ \sum_{j=1}^{m} c^{2j-1} \big)
                     \\
                     &= 
                     c^2 \brmlt \big( c^{2m} \brmlt x + c^{2m}
                     + 2\  \sum_{j=1}^{m-1} c^{2j}
                      - 2\ \sum_{j=1}^{m} c^{2j-1} \big)
                      \\
                      &{\ \ } + c^2 - 2c
                      \\
                      &=
                      c^{2m+2} \brmlt x + c^{2m+2}
                      + \sum_{j=1}^{m-1} c^2 \brmlt (2\ c^{2j})
                      \\
                      &{\ \ } +\sum_{j=1}^{m} c^2 \brmlt (-2\ c^{2j -1})
                      -2m\ c^2 + c^2 -2c && \text{ by \ref{brcmlt-ids}}
                      \\
                      &=
                      c^{2(m+1)} \brmlt x + c^{2(m+1)}
                      + \sum_{j=1}^{m-1} ( 2\ c^{2(j+1)}- c^2)
                      \\
                      &{\ \ } +\sum_{j=1}^{m} ( 3c^2 - 2c^{2j+1})
                      + (-2m + 1)\ c^2 -2c
                      && \text{ by \ref{brcmlt-ids}}
                      \\
                      &=
                      c^{2(m+1)} \brmlt x + c^{2(m+1)}
                      + 2 \sum_{j=2}^{m} ( c^{2j})
                      -2 \sum_{j=2}^{m+1} ( c^{2j-1})
                      \\
                      &{\ \ } + [-(m -1) + 3m + (-2m+1)]\ c^2 -2c
                      \\
                      &=
                      c^{2(m+1)} \brmlt x + c^{2(m+1)}
                      + 2\  \sum_{j=1}^{m} c^{2j}
                      - 2\ \sum_{j=1}^{m+1} c^{2j-1}
                     \\
                     &=
                     \tildek_{m+1}(x),
\end{align*}
as desired.

\item 
We define $\hatk (x) = k_2 k_1 (x)$ and one can show, very similarly to the calculation above, that $\hatk_m = \hatk ^m$.

\item 
We have 
\[
l_{m,n} (x) = m\ k_1(x) + m\ k_2(x) + n\ x.
\]
and then substitute $2\ c\brmlt x = 2c + 2x$.

\end{enumerate}

\end{proof}

%% file: reflections_for_factorizable_rings.tex
\section{Reflections for Factorizable Rings} \label{section-reflections-for-factorizable-rings}
In this section we obtain solutions to the reflection equation that apply specifically to factorizable rings, \ie\ nilpotent rings whose adjoint group factorises. As the next proposition shows, such nilpotent rings are braces also with a multiplication that is not the adjoint ring multiplication.

\begin{prop}[Proposition 5.2, \cite{smokt-smokt}] \label{fctr-brc}
Let \((N,+,\rngmlt)\) be a nilpotent ring and let $ (N,+,\brmlt) $ be the brace with adjoint multiplication. Assume
\(
N = B \brmlt C
\) for two subgroups \(B,C\) of \((N,\brmlt)\). Then define for 
\(
x, y \in N 
\)
where 
\(
x = x_1 \brmlt x_2
\)
and
\(
x_1 \in B,\ x_2 \in C,
\)
the operation
\[
x \fctrmlt y := x_1 \brmlt y \brmlt x_2.
\]
Then \( (N,+,\fctrmlt) \) is a brace.
\end{prop}

We can find such $B,C$ explicitly in terms of the ring operations:
\begin{prop}[Proposition 2.8, \cite{agata-leandro}] \label{nilp-rng-fctr}
Let $N$ be a nilpotent ring, $S$ be a subring of $N$ and $I$ an ideal of $N$ such that $S \cap I = \{0\}$ and the additive group factorises $N=S+I$. Then the adjoint group $(N,\brmlt)$ factorises exactly $N= S \brmlt I$.
\end{prop}




We recall the following theorem from \cite{refl}:
\begin{thm}[Theorem 2.1, \cite{refl}] \label{refl-eq-thm-sols-equiv}
Let $\abrace$ be a left brace, $r$ be the Yang-Baxter map of $\abrace$ and $X\subseteq B$ such that $(X,r)$ is a solution to the YBE. 

Let $Y\subseteq B$ and $g\colon X \to Y$ be a map. Assume there is a map $ \wedge : X \times Y \to X,$ $(x,y) \mapsto x\wedge y$ such that
\begin{equation} \label{fst-eq-thm-sols-equiv}
    \sigma_x(y \wedge g(z)) = \sigma_x(y) \wedge g(z),\ x,y,z \in X.
\end{equation}
Let $f\colon X \to X$ be a $\mathcal{G} (X,r)-$equivariant map and $k\colon X\to X$ be 
\begin{equation} \label{snd-eq-thm-sols-equiv}
    k(x) = f(x) \wedge g(x).
\end{equation}
Then $k$ is a reflection of $(X,r)$ if and only if
\begin{equation} \label{thrd-eq-thm-sols-equiv}
    f(x) \wedge g(\tau_{k(y)}(x)) = f(x) \wedge g(\tau_y(x)),\ x,y \in X.
\end{equation}
\end{thm}

The following theorem is a consequence of theorem \ref{refl-eq-thm-sols-equiv} for nilpotent rings.
\begin{thm}\label{rfl-fctr-prp}
Let $(N,+,\rngmlt)$ be a nilpotent ring, assume $N= B\brmlt C$ for some $B,C \leq (N,\brmlt)$ and $\nbrace,\nfctr$ are the braces as above. Let $r$ be the Yang-Baxter map of $\nfctr$ and let $X \subseteq N$ be such that $(X,r)$ is a solution to the YBE.

Let $f\colon X\to X$ be a \gequiv\ map and
$g\colon X \to X$ be a function such that
\begin{equation} \label{fst-eq-prop}
f(x) \rngmlt g(\tau_{k(y)}(x)) = f(x) \rngmlt g ( \tau_y(x) ) 
\end{equation}
and
\begin{equation} \label{snd-eq-prop}
g(x) \rngmlt c = c \rngmlt g(x) 
\end{equation}
for all $x,y \in X,\ c \in C$. Then 
\[
k_1(x) = f(x) \rngmlt g(x),\ k_2(x) = f(x) + f(x) \rngmlt g(x)
\]
are reflections of $(X,r)$ assuming $k_1(X), k_2(X) \subseteq X$.
\end{thm}

\begin{proof}
To show $k_1$ is a reflection of $(X,r)$, set $\wedge = \rngmlt$ in theorem \ref{refl-eq-thm-sols-equiv} so that $k_1$ coincides with the reflection $k$ of theorem \ref{refl-eq-thm-sols-equiv}. 
Then equation \eqref{thrd-eq-thm-sols-equiv} of theorem \ref{refl-eq-thm-sols-equiv} is precisely equation \eqref{fst-eq-prop} of the proposition we want to prove. 
Hence it suffices to show equation \eqref{fst-eq-thm-sols-equiv} from theorem \ref{refl-eq-thm-sols-equiv} is satisfied. 
That is, for $k_1$ to be a reflection of $(X,r)$ it suffices to show
\begin{equation*}
    \sgm{x}{y\rngmlt g(z)} = \sgm{x}{y} \rngmlt g(z),\ x,y,z \in X, \label{prop-k1}
\end{equation*}
where $\sigma_x$ is such that $r(x,y) = (\sigma_x(y), \tau_y(x)),\ x,y \in X$.

For 
$k_2(x)= f(x) + f(x) \rngmlt g(x)$ 
we set 
$ x \wedge y = x + x \rngmlt y$
in theorem \ref{refl-eq-thm-sols-equiv}.
It is easy to see that equation \eqref{thrd-eq-thm-sols-equiv} of theorem \ref{refl-eq-thm-sols-equiv} is equivalent to \eqref{fst-eq-prop} in the proposition under consideration. Thus again it suffices to prove equation \eqref{fst-eq-thm-sols-equiv} of theorem \ref{refl-eq-thm-sols-equiv}.
For $k_2$ this reduces to the same equation as for $k_1$, as the following calculations show:
\begin{alignat*}{2}
\sgm{x}{y + y \rngmlt g(z)} &= \sgm{x}{y} + \sgm{x}{y} \rngmlt g(z) \Leftrightarrow \\
\sgm{x}{y} + \sgm{x}{y \rngmlt g(z)}  &= \sgm{x}{y} + \sgm{x}{y} \rngmlt g(z) 
\Leftrightarrow \\
\sgm{x}{y \rngmlt g(z)} &=  \sgm{x}{y} \rngmlt g(z),
\end{alignat*}
where we used that $\sigma_x$ is a $(X,+)$ homomorphism by \ref{sigma-homo}. 

We now prove that 
$
\sgm{x}{y \rngmlt g(z)} =  \sgm{x}{y} \rngmlt g(z),\ x,y,z \in X.
$
Let $x = b \brmlt c,\ b \in B,\ c \in C$. For any $w\in X$ we have
\begin{alignat*}{2}
\sgm{x}{w} &= x \fctrmlt w - x \\
           &= b \brmlt w \brmlt c - b \brmlt c \\
           &= b + c + w + b\rngmlt w + b \rngmlt c + w \rngmlt c + b \rngmlt w \rngmlt c - b - c - b\rngmlt c \\
           &= w + b \rngmlt w + w \rngmlt c + b \rngmlt w \rngmlt c.
\end{alignat*}

For $ w = y \rngmlt g(z) $, using that \(c \rngmlt g(z) = g(z) \rngmlt c \) by \eqref{snd-eq-prop}, the above yields
\begin{alignat*}{2}
\sgm{x}{y \rngmlt g(z)} &= y \rngmlt g(z) + b \rngmlt (y \rngmlt g(z)) + 
                          (y \rngmlt g(z)) \rngmlt c + b \rngmlt (y \rngmlt g(z)) \rngmlt c \\
                        &=  y \rngmlt g(z) + b \rngmlt y \rngmlt g(z) + 
                          y \rngmlt c \rngmlt g(z) + b \rngmlt y \rngmlt c \rngmlt g(z) \\
                        &= (y  + b \rngmlt y  + 
                          y \rngmlt c  + b \rngmlt y \rngmlt c )\rngmlt g(z) \\
                        &= \sgm{x}{y} \rngmlt g(z),
\end{alignat*}
where in the last equality we used the calculation for $\sigma_x(w)$ above.
\end{proof}

We can recast assumption~\eqref{fst-eq-prop} in terms of the ideals of $\nring$, leading to the following theorem:
\begin{thm} \label{refl-thm-fctr-1}
Let $(N,+,\rngmlt)$ be a nilpotent ring with a subring $S$ and a two-sided ideal $I$ such that $S \cap I = 0,\ N = S + I$. 
Let $(N,+,\brmlt)$ be as above. Then $N=S \brmlt I$ and let $(N,+,\fctrmlt)$ be the brace as in~\ref{fctr-brc} with Yang-Baxter map $r$. 
Let $X \subseteq N$ be such that $(X,r)$ is a solution to the YBE.

Let $J\subseteq I \cap X$ be a two-sided ideal of the ring $\nring$. 
Let $f,g: X \to X$ such that $f$ is \gequiv.
Let $k_1(x) = f(x) \rngmlt g(x)$ and $k_2(x) = f(x) +f(x) \rngmlt g(x)$ and assume $k_1(X), k_2(X)\subseteq X$. 

If $g$ satisfies
\begin{equation} \label{ref-thm-fctr:assm1}
    g(x) \ast z  = z \ast g(x),\ z\in I,\ x \in X 
\end{equation} 
and
\begin{equation} \label{ref-thm-fctr:assm2}
    g(x + J) = g(x),\ x\in X, 
\end{equation}
and moreover $k_1(x) - x, k_2(x) -x \in J$ for $x\in X$, then $k_1, k_2$ are reflections of $(X,r)$.
\end{thm}

To prove~\ref{refl-thm-fctr-1} we need the following technical lemma:
\begin{lemma} \label{refl-lemma-factorised}
Let $(N,+,\rngmlt)$ be a nilpotent ring with a subring $S$ and a two-sided ideal $I$ such that $S \cap I = 0,\ N = S + I$. Let $(N,+,\brmlt)$ be as above. Then $N=S \brmlt I$ and let $(N,+,\fctrmlt)$ be the brace as in~\ref{fctr-brc} with Yang-Baxter map $r$. 
Let $X \subseteq N$ be such that $(X,r)$ is a solution to the YBE.

Assume $J\subseteq I \cap X$ is a two-sided ideal of the ring $\nring$ and $k: X \to X$ is such that 
\[
k(x) - x \in J,\ x\in X.
\]
Then
\[
\tau_{k(y)} - \tau_y(x) \in J,\ x,y\in X.
\]
\end{lemma}
\begin{proof}
Firstly notice that it follows immediately that $N= S\brmlt I$ from \ref{nilp-rng-fctr}, so $(N,+,\fctrmlt)$ is indeed a brace by \ref{fctr-brc}.

Write $r(x,y) = (\sigma_x(y), \tau_y(x)),\ x,y \in N$ for the Yang-Baxter map associated to $(N,+,\fctrmlt).$

The equation for $\tau$ involves a multiplicative inverse, which makes it hard to work with. So instead we will exploit the structure of the factorised ring and the properties of the ideals as much as we can using $\sigma$ and then we will relate this information to $\tau$ with \eqref{brc-sigma-tau} which says that
\[
    \sigma_x(y) \fctrmlt \tau_y(x) = x \fctrmlt y,\ x,y\in X.
\]

Let $x,y\in X$ be arbitrary and assume $\sgm{x}{y} = s \brmlt i,\ \sgm{x}{k(y)} = s\prm \brmlt i \prm$ and $x = s \dprm \brmlt i \dprm$ for $s,s\prm, s\dprm \in S$ and $i, i\prm, i\dprm \in I$. 
We expand $\sigma_x(y),\ \sigma_x(k(y))$:
\newcommand{\sdp}{s \dprm}
\newcommand{\idp}{i \dprm}
\begin{alignat*}{2}
\sgm{x}{y} &= x \fctrmlt y - x\\
           &= \sdp \brmlt y \brmlt \idp - \sdp \brmlt \idp \\
           &= \sdp + y + \idp + \sdp \rngmlt y + \sdp \rngmlt \idp + y \rngmlt \idp 
              + \sdp \rngmlt y \rngmlt \idp - \sdp -\idp -\sdp \rngmlt \idp\\
           &= y + \sdp \rngmlt y + y \rngmlt \idp + \sdp \rngmlt y \rngmlt \idp.
\end{alignat*}
Similarly
\[
\sgm{x}{k(y)} = k(y) + \sdp \rngmlt k(y) + k(y) \rngmlt \idp + \sdp \rngmlt k(y) \rngmlt \idp.
\]
\newcommand{\spr}{s \prm}
\newcommand{\ipr}{i \prm}
Hence
\begin{alignat*}{2}
s \brmlt i - \spr \brmlt \ipr &= \sgm{x}{y} - \sgm{x}{k(y)}
                              \\
                              &= (y-k(y)) + \sdp \rngmlt ( y - k(y) ) 
                              \\
                            &{\ \ } + ( y - k(y) ) \rngmlt \idp + \sdp \rngmlt ( y - k(y) ) \rngmlt \idp,
\end{alignat*}
which implies \(s \brmlt i - \spr \brmlt \ipr \in J \) since \( y - k(y) \in J \) and $J$ is a two-sided ideal of the ring $J$. Moreover, notice that
\begin{alignat*}{2}
s \brmlt i - \spr \brmlt \ipr 
&=
\underbrace{(s - \spr)}_{\in S} + 
\underbrace{(i- \ipr + s \rngmlt i - \spr \rngmlt \ipr)}_{ \in I}\\
&\in J \subseteq I
\end{alignat*}
whence $s=\spr$, because the additive group of $N$ factorises as $N= S+I$. So $\sgm{x}{k(y)} = s \brmlt \ipr$. We also obtain 
\(
(i - \ipr + s \rngmlt i - s \rngmlt \ipr) \in J.
\)

If we embed $N$ into a ring $N^1$ with multiplicative identity $1_{\rngmlt}$ the latter gives
\begin{alignat*}{2}
J &\ni i - \ipr + s \rngmlt i - s \rngmlt \ipr \\
  &= (1_{\rngmlt} + s) \rngmlt ( i - \ipr ).
\end{alignat*}
But $s+1_{\rngmlt}$ has a $\rngmlt$-inverse in $N^1$ of the form 
$\alpha +1_{\rngmlt},\ \alpha \in N$; 
multiplying by its inverse yields
\(
( i - \ipr ) \in J,
\)\ 
since if 
\[
(1_{\rngmlt} + s) \rngmlt ( i - \ipr ) = j \in J
\]
then
\[
(\alpha+1_{\rngmlt}) \rngmlt j = \alpha \rngmlt j + j \in J.
\]
We claim this further implies
\(
i^{-1} - i^{\prime - 1} \in J
\)
, where the exponent \(-1\) denotes the inverse with respect to \( \brmlt \), not $\fctrmlt$. To see this, notice that because $N$ is a nilpotent ring, for $r\in N$ we have
\[
r^{-1} = \sum_{n=1}^{\infty} (-1)^{n} r^{n}.
\]
Hence
\begin{alignat*}{2}
i^{-1} - i^{\prime - 1} &= \sum_{n=1}^{\infty} (-1)^n\ (i^n - i^{\prime\ n})\\
                        &= \sum_{n=1}^{\infty} (-1)^n\ (i-i^{\prime}) \rngmlt 
                            \Big[\   \sum_{k=0}^{n-1} \ i^k \rngmlt i^{\prime\ n-1-k}    \ \Big] \\
                        &\in J,\text{ since } (i-i^{\prime}) \in J,
\end{alignat*}
which proves our claim.

Therefore, we have shown that 
$s=s\prm$ 
and 
$i^{-1} - i^{\prime - 1} \in J$.
We will use these two statements to finish the proof.

Since 
$(N,+,\fctrmlt)$ 
is a brace, we know from proposition~\ref{sigma-homo} that
\[
\sgm{x}{y} \fctrmlt \tauu{y}{x} = x \fctrmlt y,\ x,y \in X
\]
hence
\[
s \brmlt \tauu{y}{x} \brmlt i = \sdp \brmlt y \brmlt \idp 
\]
which implies 
\[
\tauu{y}{x} = s^{-1} \brmlt \sdp \brmlt y \brmlt \idp \brmlt i^{-1}.
\]
Similarly
\[
\tauu{k(y)}{x} = s^{-1} \brmlt \sdp \brmlt k(y) \brmlt \idp \brmlt i^{\prime -1},
\]
where we used that $s\prm = s$.

Let
\[
a=s^{-1} \brmlt \sdp,\ w = y,\ b= \idp,\ c=i^{-1}
\]
\newcommand{\jp}{j\prm}
and let
\(
j, \jp \in J
\)
be such that
\[
k(y) = j+y \text{\ \ and\ \ } i^{\prime\ -1} = i^{-1} + j\prm.
\]
The existence of $j$ is guaranteed by the assumption $k(y) -y \in J$ and the existence of $j\prm$ by $i^{-1} - i^{\prime - 1} \in J$ which was proved above.

Then 
\begin{alignat*}{2}
\tauu{y}{x} - \tauu{k(y)}{x} &= 
           s^{-1} \brmlt \sdp \brmlt y \brmlt \idp \brmlt i^{-1} 
           - s^{-1} \brmlt \sdp \brmlt k(y) \brmlt \idp \brmlt i^{\prime -1} \\
           &= a\brmlt w \brmlt b \brmlt c - a \brmlt (w + j ) \brmlt b \brmlt (c+\jp)
\end{alignat*}
and
\(
\tauu{y}{x} - \tauu{k(y)}{x}  \in J
\)
follows directly from the next claim:
\begin{claim*}
Let $(R,+,\rngmlt)$ be a Jacobson radical ring and let $(R,+,\brmlt)$ be the brace with the adjoint multiplication. Let $J$ be a two-sided ideal of $R$. Moreover let $a,w,b,c \in R $ and $j,\jp \in J$ be arbitrary. Then
\[
a\brmlt w \brmlt b \brmlt c - a \brmlt (w + j ) \brmlt b \brmlt (c+\jp) \in J.
\]
\end{claim*}

\begin{proof}
We have
\begin{alignat*}{3}
&a\brmlt w \brmlt b \brmlt c - a \brmlt (w + j ) \brmlt b \brmlt (c+\jp)\\
=\
&a \brmlt \underbrace{ \big(  w \brmlt b \brmlt c - (w+j)\brmlt b \brmlt (c+\jp) \big)}_{=A} - a &&{\text{\ \ \  by \ref{cedo-lemma}}}
\\
=\
&a + A + a \rngmlt A - a 
\\
=\
& A + a \rngmlt A.
\end{alignat*}

Thus it suffices to prove 
\(
A \in J.
\)
We have
\begin{alignat*}{3}
A &= w \brmlt b \brmlt c - (w+j)\brmlt b \brmlt (c+\jp) 
  \\
  &= w \brmlt b \brmlt c - [w\brmlt b \brmlt (c+\jp) + j \brmlt b \brmlt (c+\jp) - b \brmlt (c+\jp)]y
  {\text{\ \ \ by \eqref{right-brc-dstr}}}&&
  \\
  &= w \brmlt b \brmlt c  - \big(w\brmlt b \brmlt c + w \brmlt b \brmlt \jp - w\brmlt b 
  \\
  &{\ \ \ \ \ \ \ \ \ \ \ \ \ \ \ \ \ \ \ } 
  +j \brmlt b \brmlt c + j \brmlt b \brmlt \jp - j \brmlt b\\
  &{\ \ \ \ \ \ \ \ \ \ \ \ \ \ \ \ \ \ \ }  
  -b \brmlt c - b\brmlt \jp + b     \big) 
  {\text{\ \ \ by \eqref{dstr}}}&&
  \\
  &= - w \brmlt b \brmlt \jp + w\brmlt b -  j \brmlt b \brmlt c - j \brmlt b \brmlt \jp + j \brmlt b + b \brmlt c + b\brmlt \jp - b.
\end{alignat*}
We expand each term in terms of \(\rngmlt\). The terms with the same colour cancel out and the remaining ones are in \(J\):

\begin{alignat*}{2}
-\ w \brmlt b \brmlt \jp &= \mathcolor{blue}{-w} 
                             \mathcolor{magenta}{-b} 
                             -\jp 
                             \mathcolor{red}{-w \rngmlt b} 
                             - w\rngmlt \jp 
                             - w\rngmlt b \rngmlt \jp\\
+\  w\brmlt b &=  \mathcolor{blue}{w} 
                   + \mathcolor{magenta}{b} 
                   + \mathcolor{red}{w\rngmlt b}\\
-\  j \brmlt b \brmlt c &= -j 
                            \mathcolor{magenta}{-b}
                            \mathcolor{green}{-c}
                            -j\rngmlt b
                            -j \rngmlt c
                            \mathcolor{cyan}{-b \rngmlt c}
                            - j \rngmlt b \rngmlt c\\
-\ j \brmlt b \brmlt \jp &= -j 
                            \mathcolor{magenta}{-b}
                            -\jp
                            -j\rngmlt b
                            -j \rngmlt \jp 
                            - b \rngmlt \jp 
                            -j \rngmlt b \rngmlt \jp\\
 +\ j \brmlt b &= j
                  \mathcolor{magenta}{+b}
                  +j \rngmlt b\\
 +\ b \brmlt c &= \mathcolor{magenta}{b}
                   \mathcolor{green}{+c}
                   \mathcolor{cyan}{+b\rngmlt c}\\
 +\ b\brmlt \jp &= \mathcolor{magenta}{b}
                    + \jp 
                    + b \brmlt \jp\\
 -\ b &= \mathcolor{magenta}{-b}
\end{alignat*}
Since the only remaining terms are in \(J\), \(A \in J\), as desired.
\end{proof}

Nilpotent rings are Jacobson radical, hence the above claim applies. Thus \(
\tauu{y}{x} - \tauu{k(y)}{x} \in J,
\)
as desired.
\end{proof}

\begin{proof}[Proof of theorem~\ref{refl-thm-fctr-1}]
We will apply theorem~\ref{rfl-fctr-prp}. Due to assumption~\eqref{ref-thm-fctr:assm1}, it suffices to show~\eqref{fst-eq-prop}, \ie\
\[
f(x) \rngmlt g(\tau_{k_1(y)}(x)) = f(x) \rngmlt g ( \tau_y(x) ).
\]

Since $k_1(x)-x \in J$, by lemma~\ref{refl-lemma-factorised} we have $\tau_{k_1(y)}(x) - \tau_y(x) \in J$.
Then from assumption~\eqref{ref-thm-fctr:assm1} we deduce that
$g(\tau_{k_1(y)}(x))  = g ( \tau_y(x) )$.
Then~\eqref{fst-eq-prop} follows.

The same argument works also for $k_2$.
\end{proof}

By avoiding using theorem \ref{rfl-fctr-prp}, we can prove a more general theorem for factorizable rings using directly the assumptions of \ref{refl-eq-thm-sols-equiv} and adapting some so that lemma \ref{refl-lemma-factorised} applies:

\begin{thm} \label{refl-thm-fctr-general}
Let $(N,+,\rngmlt)$ be a nilpotent ring with a subring $S$ and a two-sided ideal $I$ such that $S \cap I = 0,\ N = S + I$. Let $(N,+,\brmlt)$ be as above. Then $N=S \brmlt I$ and let $(N,+,\fctrmlt)$ be the brace as in~\ref{fctr-brc} with Yang-Baxter map $r$. 
Let $X \subseteq N$ be such that $(X,r)$ is a solution to the YBE.

Assume $J\subseteq I \cap X$ is a two-sided ideal of the ring $\nring$. 
Let $f,g: X \to X$ and $\wedge : X \times X \to X $ be maps. 
Let $k(x) = f(x) \wedge g(x)$. If all of the following are true
\begin{itemize}
    \item $f$ is \gequiv
    \item $\sigma_x(y \wedge g(z)) = \sigma_x(y) \wedge g(z)$
    \item $g(x+J) = g(x),\ x \in X$
    \item $k(x) - x \in J,\ x\in X$
\end{itemize}
Then $k(x) = f(x) \wedge g(x)$ is a reflection of $(X,r)$.
\end{thm}
\begin{proof}
By theorem \ref{refl-eq-thm-sols-equiv} it suffices to show
\begin{equation} \label{eq-to-show}
    f(x) \wedge g(\tau_{k(y)}(x)) = f(x) \wedge g(\tau_y(x)),\ x,y \in X,
\end{equation}
as all other assumptions of the theorem are trivially satisfied.

Since by assumption $k(x) - x \in J,\ x\in X$, lemma \ref{refl-lemma-factorised} yields
\[
\tau_{k(y)} (x) - \tau_y(x) \in J,\ x \in X
\]
which implies
\[
g(\tau_{k(y)} (x)) = g(\tau_y(x)),\ x,y \in X
\]
by the assumption $g(x+J)=g(x)$. Then equation \eqref{eq-to-show} follows.
\end{proof}

\begin{prop} \label{fctr-thm-socle}
Let $(N,+,\rngmlt)$ be a nilpotent ring with a subring $S$ and a two-sided ideal $I$ such that $S \cap I = 0,\ N = S + I$. 
Let $(N,+,\brmlt)$ be as above. Then $N=S \brmlt I$ and let $(N,+,\fctrmlt)$ be the brace as in~\ref{fctr-brc} with Yang-Baxter map $r$. 
Let $X \subseteq N$ be such that $(X,r)$ is a solution to the YBE.

Assume $J\subseteq I \cap X$ is a two-sided ideal of the ring $\nring$ and $k: X \to X$ is such that for each $x\in X$ it satisfies
$\inlineequation[fctr-thm-socle:as1]{
k(x+J)=k(x)
},$
$\inlineequation[fctr-thm-socle:as2]{
k(x) - x \in J
}$
and
$\inlineequation[fctr-thm-socle:as3]
{
k(x) \in \soc{B^\op} \cap X
}.$
Then $k$ is a reflection of $(X,r)$.
\end{prop}

\begin{proof}
Let $x,y\in X$.
By theorem 1.8 of~\cite{refl} it suffices to show that the first coordinates of the reflection equation match, \ie
\begin{equation} \label{fctr-thm-scole:to-show}
    \sigma_{\sigma_x(y)} k \tau_y(x) = \sigma_{\sigma_x k(y)} k \tau_{k(y)} (x),\ x,y \in X.
\end{equation}

From assumption~\eqref{fctr-thm-socle:as2} and lemma~\ref{refl-lemma-factorised} it follows that
\[
\tau_{k(y)} (x) - \tau_y(x) \in J
\]
whence by~\eqref{fctr-thm-socle:as1}
\[
k \tau_{k(y)} = k \tau_y(x).
\]
Moreover, by assumption~\eqref{fctr-thm-socle:as3} both sides of the equality lie in $\soc{B^\op} \cap X$ and hence by proposition~\ref{soccle-fixpoint} they are fixpoints of $\sigma_z$ for each $z\in X$.
Thus~\eqref{fctr-thm-scole:to-show} is true and we are done.
\end{proof}

%% file: parameter_dependent.tex
\section{The Parameter-dependent Reflection Equation} \label{section-parameter-dependent}

In the theory of integrable systems~\cite{sklyanin,fst-refl,klimyk}, the version of the YBE that is used is the \emph{parameter-dependent} version:
\begin{equation} \label{ybe-param-eq}
    (R(u) \otimes I)(I \otimes R(u+v))(R(v) \otimes I)
    =
    (I\otimes R(v)) (R(u+v) \otimes I)(I \otimes R(u))
\end{equation}
where $R : \complex \to (V\otimes V \to V\otimes V)$ for some vector space $V$ and \eqref{ybe-param-eq} must be true for all $u,v\in \complex$. 

The same is true for the reflection equation:
\begin{dfn}[pp. 368, equation (89), \cite{klimyk}]
Let $(V,R)$ be a solution to the parameter-dependent YBE. We say that $K: \complex \to (V \to V)$ satisfies the \emph{parameter-dependent reflection equation} if for all $u,v \in \complex$
\begin{multline} \label{refl-param-dep}
    (I \otimes K(v)) (R(u+v)) (I\otimes K(u)) (R(u-v))\\
    =
    (R(u-v)) (I\otimes K(u)) (R(u+v))(I \otimes K(v))
\end{multline}
\end{dfn}

The following theorem from \cite{refl} shows how involutive solutions to the set-theoretic reflection equation extend to solutions of the parameter-dependent equation~\eqref{refl-param-dep}. We remind the reader that for a set $X$ and some map $f:X\to X$ we can extend $f$ to $F: \complex(X) \to \complex(X)$, where $\complex(X)$ is the $\complex$-vector space with basis $X$, by defining $F(x) = f(x),\ x\in X$ and then extending linearly to linear combinations of $x\in X$. It is easy to see that $\complex (X\times X) = \complex(X) \otimes \complex(X)$.

\begin{thm}[Theorem 5.1, \cite{refl}] \label{invol-param-dep-refl}
Let $(X,r)$ be an involutive non-degenerate solution to the set-theoretic YBE and $k$ a reflection of $(X,r)$. 

Let 
$V=\complex(X)$ and let $K$ be the linear extention of $k$ to
$V\to V$ and $R$ the linear extension of $r$ to $V\otimes V \to V\otimes V$. 

Define $K\prm : \complex \to (V \to V)$ and $R\prm : \complex \to (V\otimes V \to V\otimes V)$ by
\[
K\prm (u) (x) = uK(x),\ u\in \complex, x\in V
\]
and
\[
R\prm (u) (x) = I + uR(x),\ u\in \complex, x\in V.
\]
If we assume that $k$ is involutive, then $K\prm$ satisfies the parameter-dependent reflection equation for $(V,R\prm)$.
\end{thm}


In the next theorem we show that two special solutions of theorem~\ref{k1-k2-equiv} are involutive and thus by theorem~\ref{invol-param-dep-refl} yield solutions to the parameter-dependent reflection equation.

\begin{thm}
Let $\abrace$ be a left brace, $r$ the Yang-Baxter map of $\abrace$ and $X\subseteq B$ such that $(X,r)$ is a solution to the YBE. 

Let $c\in B$ be central in $(B,\brmlt)$. 

The map
\[
k_1(x) = c\brmlt x -c,\ x\in X,
\]
assuming $k_1(X) \subseteq X$, is an involutive reflection of $(X,r)$ if $c^2 \in \soc{B}$. In particular, it is an involutive reflection if $c$ is involutive in $(B,\brmlt)$.

If we assume that $2\ c\brmlt x = 2c+2x,$ for all $x\in X$, and $c$ is involutive, then the map
\[
k_2(x) = c\brmlt x +c,\ x\in X
\]
is an involutive reflection of $(X,r)$, assuming $k_2(X) \subseteq X$.
\end{thm}
\begin{proof}
The maps $k_1,k_2$ are special cases of the maps in theorem~\ref{k1-k2-equiv}, where we showed they are \gequiv\ and thus reflections of $(X,r)$ by theorem~\ref{equiv-refl}. Thus it suffices to prove that $k_1^2 = k_2^2 = \id$.

Let $x\in X$. For $k_1$ we have
\begin{align*}
    k_1^2(x) &= k_1 (c\brmlt x - c)\\
             &= c\brmlt (c\brmlt x -c) - c\\
             &= c^2 \brmlt x - c^2 &&{\text{  by \ref{cedo-lemma}}} \\
             &= x &&{\text{  $c^2\in \soc{B}$}}
\end{align*}
thus $k_1 ^2 = \id$. In particular, if $c$ is involutive then $c^2 = 0 \in \soc{B}$.

For $k_2$ we have
\begin{align*}
    k_2^2(x) &= k_2(c\brmlt x + c)\\
             &= c\brmlt (c\brmlt x + c) + c\\
             &= c\brmlt c\brmlt x + c\brmlt c &&{\text{  by \eqref{dstr}}}&& \\
             &= 0 \brmlt x + 0 &&{\text{  $c$ is involutive in $(B,\brmlt)$}}&&\\
             &= x
\end{align*}
hence $k_2^2 = \id$ as desired.

\end{proof}

The next theorem shows that one reflection of the brace $\nbrace$ from theorem~\ref{refl-thm-fctr-1} is involutive and thus yields a solution to the parameter-dependent reflection equation.
\begin{thm}
Let $(N,+,\rngmlt)$ be a nilpotent ring with a subring $S$ and a two-sided ideal $I$ such that $S \cap I = 0,\ N = S + I$.
Let $(N,+,\brmlt)$ be as above. Then $N=S \brmlt I$ and let $(N,+,\fctrmlt)$ be the brace as in~\ref{fctr-brc} with Yang-Baxter map $r$. 
Let $X \subseteq N$ be such that $(X,r)$ is a solution to the YBE.

Assume $J\subseteq I \cap X$ is a two-sided ideal of the ring $\nring$. Let $g: X\to X$ be a map that for each $x\in X$ satisfies
$g(x) \ast z  = z \ast g(x),\ z\in I,$
$g(x + J) = g(x) $
and
$g(x) \in J$. 
Moreover assume that for every $x\in X$, $g(x)$ is involutive in the adjoint group $(N,\brmlt)$.
Then $k(x) = x+x\rngmlt g(x)$ is an involutive reflection of $(X,r)$.
\end{thm}
\begin{proof}
The first three assumptions guarantee that $k$ is a reflection of $(X,r)$ by theorem~\ref{refl-thm-fctr-1} for $k_2=k$ by setting $f(x) = x$. Thus it suffices to show $k^2(x) = x$.
We calculate $k^2(x)$:
\begin{align*}
    k^2(x) &= k(x) + k(x) \rngmlt g( k(x) ) \\
           &= x + x\rngmlt g(x) + 
             (x+x\rngmlt g(x)) \rngmlt g(x + 
                   \underbrace{x\rngmlt g(x)}_{\in J}) 
                   &&\text{ since $g(x) \in J$} \\
           &= x + x\rngmlt g(x) + x\rngmlt g(x) + x\rngmlt g(x) \rngmlt g(x) 
           &&\text{$g(x+J) = g(x)$}\\
           &= x + x\rngmlt (g(x) + g(x) + g(x) \rngmlt g(x))\\
           &= x + x\rngmlt ( g(x) \brmlt g(x))\\
           &= x + x\rngmlt 0\\
           &= x,
\end{align*}
as desired.
\end{proof}
